\documentclass[10pt]{amsart}

\usepackage{geometry}
\usepackage{caption}
\usepackage[hang,flushmargin]{footmisc}
\usepackage[colorlinks,citecolor=red,pagebackref,hypertexnames=false]{hyperref}

\usepackage{amssymb,amsmath,amsthm,amsfonts,amsrefs,microtype}
\usepackage{mathtools}
\usepackage{mathrsfs}
\usepackage{bbm}
\usepackage{enumitem}

\usepackage{tikz,tikz-3dplot}
\usetikzlibrary{shapes,positioning,intersections,quotes}
\usepackage[font=small,labelfont=bf]{caption}

\numberwithin{equation}{section}

\theoremstyle{plain}
\newtheorem{theorem}{Theorem}[section]

\newtheorem{corollary}[theorem]{Corollary}
\newtheorem{proposition}[theorem]{Proposition}

\theoremstyle{definition}

\numberwithin{equation}{section}

\newcommand{\rnm}{\mathbb{R}}
\newcommand{\znm}{\mathbb{Z}}

\newcommand{\eun}[1]{\rnm^{#1}}

\newcommand{\ra}{\rightarrow}
\newcommand{\Ra}{\Rightarrow}

\newcommand{\LRA}{\Longleftrightarrow}

\newcommand{\se}{\subseteq}

\DeclareMathOperator{\sspan}{span}

\DeclareMathOperator{\spt}{\text{sppt}}

\newcommand{\mc}[1]{\mathcal{#1}}
\newcommand{\mbb}[1]{\mathbb{#1}}

\newcommand{\os}[1]{^{#1}}
\newcommand{\us}[1]{_{#1}}

\newcommand{\nf}{\infty}
\newcommand{\tm}{\times}
\newcommand{\pd}{\partial}
\newcommand{\dx}[1]{\,d#1}

\newcommand{\cl}{\colon}

\newcommand{\prs}[1]{\left(#1\right)}
\newcommand{\prss}[1]{(#1)}

\newcommand{\abs}[1]{\left|#1\right|}

\newcommand{\angs}[1]{\left\langle#1\right\rangle}

\newcommand{\cbrks}[1]{\left\{#1\right\}}

\newcommand{\norm}[1]{\left\lVert#1\right\rVert}
\newcommand{\norms}[1]{\lVert#1\rVert}

\newcommand{\nlspace}{\\[0.3cm]}

\newcommand\numberthis{\addtocounter{equation}{1}\tag{\theequation}}

\newcommand{\sph}[1]{\mbb{S}^{#1}}

\newcommand{\comment}[1]{}
\newcommand{\ts}{\textstyle}
\DeclareMathOperator{\conv}{Conv}
\newcommand{\parl}{\parallel}

\allowdisplaybreaks

\begin{document}
	

\title[Simplex Averaging Operators: quasi-Banach and $L^p$-improving Bounds]{Simplex Averaging Operators: quasi-Banach\\ and $L^p$-improving Bounds in Lower Dimensions}
\author{Alex Iosevich}
\author{Eyvindur Ari Palsson}
\thanks{The work of the first listed author was supported in part by NSF grant HDR TRIPODS - 1934962, the work of the second listed author was supported in part by Simons Foundation Grant \#360560, and the work of the third listed author was supported in part by NSF grant DMS - 1907435.}
\author{Sean R. Sovine}

\maketitle

\vspace*{-8mm}
\begin{abstract}
	We establish some new $L^p$-improving bounds for the $k$-simplex averaging operators $S^k$ that hold in dimensions $d \geq k$. As a consequence of these $L^p$-improving bounds we obtain nontrivial bounds $S^k\colon L^{p_1}\times\cdots\times L^{p_k}\rightarrow L^r$ with $r < 1$. In particular we show that the triangle averaging operator $S^2$ maps $ L^{\frac{d+1}{d}}\times L^{\frac{d+1}{d}} \rightarrow L^{\frac{d+1}{2d}}$ in dimensions $d\geq 2$. This improves quasi-Banach bounds obtained in \cite{PalssonSovine} and extends bounds obtained in \cite{GreenleafIosevichKrauseLiu} for the case of $k = d = 2$. 
\end{abstract}


\vspace*{1mm}

\section{Introduction}

Let $d\geq k$ and let $\Delta_k = \{ u_0 = 0, u_1, \ldots, u_k \}\se \eun{d}$ be the set of vertices of a regular $k$-simplex of unit side length. We define the {$k$-simplex averaging operator}
\[
S^{k}(f_1, \ldots, f_{k})(x) := \int_{O(d)}f_1(x - Ru_1)\cdots f_{k}(x - Ru_{k})\dx{\mu(R)},
\]
where $\mu$ is the normalized Haar measure on the group $O(d)$. At input $x$ this operator computes the average value of the function $f_1 \otimes \cdots \otimes f_k$ on the smooth manifold 
\[
\mc M_k(x) = \{ (v_1, \ldots, v_k) \in (\eun{d})^k ~\cl~ |v_i - v_j|^2 = 1 \text{ for } 0\leq i < j \leq k, \text{ with }v_0 = x \}
\]
of all tuples $(v_1, \ldots, v_k) \in (\eun{d})^k$ such that $\{x, v_1, \ldots, v_k\}$ is the set of vertices of a regular $k$-simplex of unit side length.
The $k$-simplex averaging operator is a $k$-linear analogue of the spherical averaging operator, which computes the average value of the function $f$ over a sphere centered at $x$ and can be expressed as
\[
S^1(f)(x):= \int_{O(d)}f(x - Ru_1)\dx{\mu(R)},
\]
for any $u_1$ with $|u_1| = 1$.

Cook, Lyall, and Magyar \cite{CookLyallMagyar} introduce a technique that can be used to establish a wide range of nontrivial and $L^p$-improving bounds for averages over non-degenerate $k$-simplices in higher dimensions. 
In this work we establish $L^p$ improving bounds for $S^k$ that hold in lower dimensions and show how these can be used to obtain further quasi-Banach bounds for $S^k$. 
In the case with $k=2$ we have the {triangle averaging operator}, which we denote by $T:= S^2$, and our result is:
\begin{theorem}\label{t1}
	The triangle averaging operator $T$ satisfies the bound.
	\[
	T\cl L\os{\frac{d+1}{d}}(\eun{d})\tm L\os{\frac{d+1}{d}}(\eun{d}) \ra L\os{s}(\eun{d}), \quad \text{ for all }~~~ s \in [\textstyle{\frac{d+1}{2d}}, 1] ~\text{ and }~~ d\geq 2,
	\]
	Moreover, 
	\[
	T\cl L\os{p}(\eun{d})\tm L\os{q}(\eun{d}) \ra L\os{1}(\eun{d})
	\]
	if and only if $(\frac{1}{p}, \frac{1}{q})$ lies in the convex hull of the points $\{(0,1), (1,0), (\frac{d}{d+1}, \frac{d}{d+1})\}$.
\end{theorem}
\noindent For the simplex operators $S^k$ we establish the following $L^p$ improving bounds that hold in lower dimensions. 
\begin{theorem}
	In dimensions $d\geq k$,
	\begin{align*}
	&S^k \text{ is of restricted strong-type } \ts\prs{ k, \ldots, k, k },\nlspace
	&S^k \text{ is of restricted strong-type } \ts\prs{ k\frac{d+1}{d}, \ldots, k\frac{d+1}{d}, d+1 }.
	\end{align*}
\end{theorem}
\noindent When $d$ is large the unrestricted version of the first of these bounds follows from the second bound by interpolation, but this is not the case when $d$ is close to $k$. 
In higher dimensions these bounds are contained in the range of bounds obtained by Cook, Lyall, and Magyar.
Our proof of the first bound is an adaptation of the proof given by Greenleaf, Iosevich, Krause, and Liu \cite{GreenleafIosevichKrauseLiu} in the case where $k = d = 2$, which can also be derived from the work of Stovall \cite{StovallThesis}. 

We also describe a technique for obtaining bounds into $L^r$ with $r< 1$ from $L^p$ improving bounds mapping into $L^1$.  As observed in \cite{GreenleafIosevichKrauseLiu}, for $f, g \geq 0$
\begin{align*}
	\norms{S^2(f,g)}_{L^1} = \angs{f, S^1(g)} = \angs{S^1(f), g}, 
\end{align*}
and hence by the well-known bounds for the spherical averaging operator $S^1$ we have that
\[
S^2 \cl L^p \tm L^q \ra L^1 \quad \text{ iff } \quad \ts\prs{\frac{1}{p}, \frac{1}{q}} \in \ts\conv\cbrks{(0,1), (1,0), \prss{\frac{d}{d+1}, \frac{d}{d+1}}}.
\]
This gives the bounds in Theorem \ref{t1},
which improve on bounds obtained in \cite{PalssonSovine}. As a further application of this technique we show that the bilinear spherical averaging operator
\[
B(f,g)(x):= \int_{\sph{2d-1}}f(x - u_1)g(x - u_2)\dx{\sigma(u_1,u_2)}
\]
maps $L^1 \tm L^1 \ra L^s$ for $s \in [1/2, 1]$ and $d\geq 2$. 


\section{Bounds of Cook, Lyall, and Magyar}
The following result was established in \cite{CookLyallMagyar}:
\begin{theorem}[Special case of Proposition 3 of Cook, Lyall, and Magyar \cite{CookLyallMagyar}]\label{p1}
	Let $k, m \geq 2$ be integers with $d \geq km$. Then the $k$-simplex averaging operator $S^k$ satisfies the bounds
	\[
	S^k(f_1, \ldots, f_k)(x) \leq C_{d, m, k}\prss{S\os{k-1}(|f_1|^q, \ldots, |f_{k-1}|^q)(x)}\os{\frac{1}{q}}\prs{S(|f_k|^q)}\os{\frac{1}{q}},
	\]
	uniformly for $x \in \eun{d}$, where $q = \frac{m}{m-1}$. Hence by induction,
	\[
	S^k(f_1, \ldots, f_k)(x) \leq C_{d,m,k}\prss{ S(|f_1|\os{q\os{k-1}})(x) }\os{\frac{1}{q\os{k-1}}}\prss{ S(|f_2|\os{q\os{k-1}})(x) }\os{\frac{1}{q\os{k-1}}}\prod_{j=3}^{k}\prss{ S(|f_j|\os{q\os{k+1-j}})(x) }\os{\frac{1}{q\os{k+1-j}}},
	\]
	uniformly for $x \in \eun{d}$. 
\end{theorem}
\noindent Combining this with H\"older's inequality yields the following range of $L^p$ bounds for $S^k$, which includes near-optimal non-trivial bounds.
\begin{corollary}\label{cc1}
	The operator $S^k$ satisfies the bounds $$S^{k}\cl L\os{p_{\sigma(1)}} \tm \cdots \tm L\os{p_{\sigma(k)}} \ra L^r$$
	for all exponents satisfying
	$p_1 \geq q\os{k-1}$,\, $p_j \geq q\os{k+1 - j}$ for $j \geq 2$, and $\frac{1}{p_1} + \cdots + \frac{1}{p_k} = \frac{1}{r}$, whenever $k, m \geq 2$, $d \geq mk$, and $q = \frac{m}{m-1}$, for all permutations $\sigma$ of $\{1, \ldots, k\}$. Hence by interpolation
	\[
	S^k \cl L\os{kr}\tm\cdots\tm L\os{kr} \ra L^r \quad \text{ with } \quad r = \frac{q\os{k-1}}{2 + q + q^2 + \ldots + q\os{k-2}}.
	\]
\end{corollary}
\noindent These bounds are asymptotically optimal as $m$, and hence $d$, increases. 
Combining Theorem \ref{p1} with bounds for the spherical average gives strong $L^p$-improving bounds, for example,
\begin{corollary}\label{cc2}
 The operator $S^k$ satisfies the bounds
\[
S^k\cl L\os{p_{\sigma(1)}} \tm \cdots \tm L\os{p_{\sigma(k)}} \ra L^r, \quad \text{ where } \quad r = \frac{q\os{k-1}(d+1)}{2 + q + q^2 + \ldots + q\os{k-2}},
\]
where $p_1 = q\os{k-1}\frac{d+1}{d}$ and $p_j = q\os{k+1-j}\frac{d+1}{d}$ for $j \geq 2$, for each permutation $\sigma$.
Hence by interpolation
\[
S^k\cl L\os{\frac{kr}{d}}\tm\cdots \tm L\os{\frac{kr}{d}} \ra L^r,
\]
for all $m\geq 2$ and $d\geq mk$. 
\end{corollary}


\section{Background}

The triangle averaging operator was introduced in dimension $d = 2$ by Greenleaf and Iosevich in \cite{GreenleafIosevichTriangle}, where Sobolev bounds for $T$ were obtained and applied to a generalization of the Falconer distance problem. 
Greenleaf, Iosevich, Krause, and Liu \cite{GreenleafIosevichKrauseLiu} showed that in dimension $d=2$ a family of operators including $T$ satisfies $L^p \tm L^q \ra L^r$ bounds for $(\frac{1}{p}, \frac{1}{q}, \frac{1}{r})$ in the set $\{ (\frac{2}{3}, \frac{2}{3}, 1), (\frac{2}{3}, 0, \frac{1}{3}), (0, \frac{2}{3}, \frac{1}{3}) \}$ and a restricted strong-type bound for $(\frac{1}{p}, \frac{1}{q}, \frac{1}{r}) = (\frac{1}{2}, \frac{1}{2}, \frac{1}{2})$, and showed that these $L^p$ improving bounds are sharp in the Banach range. These bounds can also be derived from the work of Stovall \cite{StovallThesis}.
In~\cite{PalssonSovine} Palsson and Sovine studied the $L^{p} \tm L\os{p} \ra L^r$ boundedness of $T(f,g)$ using a frequency-space decomposition and obtained quasi-Banach bounds in higher dimensions. Cook, Lyall, and Magyar~\cite{CookLyallMagyar} established bounds for maximal averages with respect to general non-degenerate $k$-simplices using the majorization technique described above.



\section{Quasi-Banach Bounds from $L^p$-Improving Bounds into $L^1$}\label{s1}

In \cite{GrafakosKalton} Grafakos and Kalton show that the operator
\[
I(f,g)(x) := \int_{|t|\leq 1} f(x - t) g(x + t)\dx{t}
\]
is bounded on $L^1(\eun{d}) \tm L^1(\eun{d}) \ra L\os{1/2}(\eun{d})$. 
We show how their argument can be adapted to a slightly more general situation.  
In the following, for $l = (l_1, \ldots, l_d) \in \znm^d$ we denote by $Q_l$ the cube with side length 1 and lower left corner at $l$.

Suppose that the $k$-linear operator $U(f_1, \ldots, f_k)$ 
has the following localization properties: 
\begin{enumerate}
	\item[(L1)] There is a finite number $N$ such that $U(f_1, \ldots, f_k) \equiv 0$ whenever there are $i,j$ with $f_i, f_j$ supported on cubes $Q_{l^i}, Q_{l^j}$ with $\norm{l^i - l^j}_\nf := \max_{1\leq n \leq d}|(l^i)_n - (l^j)_n| > N$.\\
	\item[(L2)] There is a fixed $R > 0$ such that $U(f_1, \ldots, f_k)(x)$ is supported on
	$
	\bigcup_{i=1}^k\spt(f_i) + B(0, R).
	$
\end{enumerate}
It is easy to see that each of the $k$-simplex averaging operators $S^k$ and the bilinear spherical averaging operator satisfy conditions L1 and L2.
Now suppose that whenever each $f_i$ is supported on a cube of side length 1 we have the bound
\[
\norm{U(f_1, \ldots, f_k)}_{L^1(\eun{d})} \leq A\norm{f_1}_{L\os{p_1}(\eun{d})}\cdots \norm{f_k}_{L\os{p_k}(\eun{d})} \numberthis\label{e1}
\]
for some exponents with 
\[
\frac{1}{p_1} + \cdots + \frac{1}{p_k} =: \frac{1}{r} > 1.
\]
We define $F_N := \{ l \in \znm^d \cl \norm{l}_\nf \leq N \}$. 
Then by properties L1 and L2 we have for each $s \in [r, 1]$,
\begin{align*}
\norm{U(f_1, \ldots, f_k)}_{L^s} &=  \prs{\int_{\eun{d}} \abs{ \sum_{l\in \znm^d}\sum_{d_2,\ldots, d_k\in F_N}
		U(f_11_{Q_l}, f_21_{Q_l + d_2} \ldots, f_k1_{Q_l + d_k})(x)}^s\dx{x} }\os{\frac{1}{s}}
\nlspace
&\leq C\sum_{d_2,\ldots, d_k\in F_N}
\prs{\int_{\eun{d}} \abs{ \sum_{l\in \znm^d}
		U(f_11_{Q_l}, f_21_{Q_l + d_2} \ldots, f_k1_{Q_l + d_k})(x)}^s\dx{x} }\os{\frac{1}{s}}
\nlspace
&\leq C\sum_{d_2,\ldots, d_k\in F_N}
\prs{\int_{\eun{d}} \sum_{l\in \znm^d} \abs{
		U(f_11_{Q_l}, f_21_{Q_l + d_2} \ldots, f_k1_{Q_l + d_k})(x)}^s\dx{x} }\os{\frac{1}{s}}
\nlspace
&\leq C\sum_{d_2,\ldots, d_k\in F_N} \prs{\sum_{l\in \znm^d} \norm{ U(f_11_{Q_l}, f_21_{Q_l + d_2} \ldots, f_k1_{Q_l + d_k}) }_{L^1}^s }\os{\frac{1}{s}}\nlspace
%
%
&\leq C\sum_{d_2,\ldots, d_k\in F_N}
\prs{ \sum_{l\in\znm^d}
	\norm{f_11_{Q_l}}^s_{L\os{p_1}}\norm{f_21_{Q_l + d_2}}^s_{L\os{p_2}}\cdots \norm{f_k1_{Q_l + d_k}}^s_{L\os{p_k}} }\os{\frac{1}{s}}\nlspace
&\leq C\sum_{d_2,\ldots, d_k\in F_N}
\prs{ \sum_{l\in\znm^d}
	\norm{f_11_{Q_l}}^r_{L\os{p_1}}\norm{f_21_{Q_l + d_2}}^r_{L\os{p_2}}\cdots \norm{f_k1_{Q_l + d_k}}^r_{L\os{p_k}} }\os{\frac{1}{r}}
\nlspace
&\leq C \prs{ \sum_{l\in \znm^d}\norm{f_11_{Q_l}}\os{p_1}_{L\os{p_1}} }\os{\frac{1}{p_1}}\cdots 
\prs{ \sum_{l\in \znm^d}\norm{f_k1_{Q_l}}\os{p_k}_{L\os{p_k}} }\os{\frac{1}{p_1}}
\nlspace
&= C\norm{f_1}_{L\os{p_1}}\cdots \norm{f_k}\us{L\os{p_k}},
%
\end{align*}
where the constant depends on $N$, $R$, $d$, $s$, and $A$. We summarize this result in the following proposition.

\begin{proposition}\label{p2}
	Suppose that the $k$-linear operator $U(f_1, \ldots, f_k)$ satisfies the localization conditions (L1) and (L2) and that $
	\norm{U(f_1, \ldots, f_k)}_{L^1(\eun{d})} \leq A\norm{f_1}_{L\os{p_1}(\eun{d})}\cdots \norm{f_k}_{L\os{p_k}(\eun{d})}
	$
	for some exponents $p_1, \ldots, p_k \geq 1$ with 
	$
	\frac{1}{p_1} + \cdots + \frac{1}{p_k} =: \frac{1}{r} > 1
	$ whenever each $f_i$ is supported on a cube. Then for each $s \in [r, 1]$
	\[
	U \cl L\os{p_1}(\eun{d}) \tm \cdots \tm L\os{p_k}(\eun{d}) \ra L\os{s}(\eun{d}). 
	\]
\end{proposition}

Note that the bilinear convolution operator $T_\mu$ associated to any compactly supported finite Borel measure on $\eun{2d}$ satisfies the localization conditions (L1) and (L2). The following proposition is an abstract version of the technique used to obtain the bound $I\cl L^1 \tm L^1 \ra L\os{\frac{1}{2}}$ of Grafakos and Kalton \cite{GrafakosKalton} and our result below on the boundedness of $T$. 
\begin{proposition}\label{p3}
	Let $\mu$ be a compactly supported finite positive Borel measure on $\eun{2d}$ such that the pushforward measure
	\[
	\mu_{(-)}(A) := \int_{\eun{2d}}1_A(y - z)\dx{\mu(y,z)}
	\]
	on $\eun{d}$ is absolutely continuous with density $\frac{d\mu_{(-)}}{dt} \in L^\nf$. Then 
	\[
	T_\mu \cl L^1(\eun{d}) \tm L^1(\eun{d}) \ra L\os{1}(\eun{d}),
	\]
	and thus by Proposition \ref{p2}
	\[
	T_\mu \cl L^1(\eun{d}) \tm L^1(\eun{d}) \ra L\os{s}(\eun{d}),
	\]
	for $s \in [\frac{1}{2}, 1]$.
	
	If $T_{\mu_{(-)}}\cl L^p \ra L^q$ with $q > p$, then
	\[
	T_\mu \cl L^{q'}(\eun{d}) \tm L^p(\eun{d}) \ra L\os{1}(\eun{d}),
	\]
	and thus by Proposition \ref{p2}
	\[
	T_\mu \cl L^{q'}(\eun{d}) \tm L^p(\eun{d}) \ra L\os{s}(\eun{d}),
	\]
	for $s \in [\frac{pq'}{p + q'}, 1]$.
\end{proposition} 

\begin{proof}
	Suppose that $\mu_{(-)}$ is absolutely continuous with $L^\nf$ density. 
	This proof is essentially the same as the one given by Grafakos and Kalton to bound $I$. We have
	\begin{align*}
	\norm{T_\mu(f,g)}_{L^1} &\leq \int_{\eun{d}}\int_{\eun{2d}}|f(x - u)||g(x - v)|\dx{\mu(u,v)}\dx{x}\nlspace
	&= \int_{\eun{d}}|f(x)|\int_{\eun{2d}}|g(x - (u - v))|\dx{\mu(u,v)}\dx{x}\nlspace
	&= \int_{\eun{d}}|f(x)|\int_{\eun{d}}|g(x - t)|\dx{\mu_{(-)}(t)}\dx{x}\nlspace
	&= \int_{\eun{d}}|f(x)|\int_{\eun{d}}|g(x - t)|\frac{d\mu_{(-)}}{dt}\dx{t}\dx{x}\nlspace
	&\leq \norms{{d\mu_{(-)}}/{dt}}_{L^\nf}\norm{f}_{L^1}\norm{g}_{L^1}. 
	\end{align*}
	We can now apply Proposition \ref{p2}.
	
	Now suppose that $T_{\mu_{(-)}}\cl L^p \ra L^q$. Then we have 
	\begin{align*}
	\norm{T_\mu(f,g)}_{L^1} &\leq \int_{\eun{d}}|f(x)|\int_{\eun{d}}|g(x - t)|\dx{\mu_{(-)}(t)}\dx{x}\nlspace
	&\leq \norm{f}_{L\os{q'}}\norm{T_{\mu_{(-)}}(g)}_{L^q}\nlspace
	&\leq C\norm{f}_{L\os{q'}}\norm{g}_{L^p}. \\[-8mm]
	\end{align*}
\end{proof}


\vspace*{-1cm}

\section{Applications to the Bilinear Spherical and Triangle Averaging Operators}\label{s2}

\subsection{Application to triangle averaging operator}

We will use the $L^p$ improving bound $S^1 \cl L\os{d + 1} \ra \frac{d+1}{d}$ for the spherical averaging operator to estimate the $L^1$ norm of $T = S^2$. By Tonelli's theorem and a change of variables we have
\begin{align*}
\norms{T(f_1, f_2)}_{L^1} &\leq \int_{SO(d)}\int_{\eun{d}} \abs{f_1(x - Ru_1)}\abs{f_2(x - Ru_2)}\dx{x}\dx{\mu(R)}\nlspace
&= \int_{\eun{d}}\abs{f_1(x)}\int_{SO(d)}\abs{f_2(x - R(u_2 - u_1))}\dx{\mu(R)}\dx{x}\nlspace
&= \norm{f_1(x)S^1(|f_2|)(x)}_{L^1}\nlspace
&\leq \norm{f_1}_{L\os{\frac{d+1}{d}}}\norm{S^1(|f_2|)}_{L\os{d+1}}\nlspace
&\leq C\norm{f_1}_{L\os{\frac{d+1}{d}}}\norm{f_2}_{L\os{\frac{d+1}{d}}}.
\end{align*}
Now it follows by Proposition \ref{p1} that
\[
T \cl L\os{\frac{d+1}{d}}\tm L\os{\frac{d+1}{d}} \ra L\os{s}
\]
for all $s \in [\frac{d+1}{2d}, 1]$. This argument was previously used in \cite{GreenleafIosevichKrauseLiu}.

In fact, the reasoning above shows that for $f, g \geq 0$, 
\[
\norm{T(f,g)}_{L^1} = \angs{S(f), g} = \angs{f, S(g)}.
\]
It follows by $L^p$ duality that
\[
T\cl L^p \tm L^q \ra L^1 \quad \text{ if and only if } \quad S\cl L^p \ra L\os{q'}. 
\]
Hence from the known range of bounds for the spherical averaging operator (see for example \cite{LaceySparseSphericalMaximal}) we have that $T\cl L^p \tm L^q \ra L^1$ if and only if $(\frac{1}{p}, \frac{1}{q})$ lies in the region shown in Figure~\ref{fig1}. Thus the bounds $T\cl L^p \tm L^q \ra L^r$ with $\frac{1}{p} + \frac{1}{q} = \frac{1}{r} \geq 1$ that can be obtained by applying Proposition \ref{p2} are exactly those with $(\frac{1}{p}, \frac{1}{q})$ in the region shown in Figure~\ref{fig1}.
The essential new $T\cl L^p \tm L^q \ra L\os{\frac{pq}{p+q}}$ bound in this range is the one with $(\frac{1}{p}, \frac{1}{q}) = (\frac{d}{d+1}, \frac{d}{d+1})$, since the others can be obtained from this one by interpolation with bounds in the Banach range.

\begin{center}
	\begin{tikzpicture}[scale=0.5]
	\draw[color=gray!50, fill=gray!20, style = solid] (6,0) -- (40/8,40/8) -- (0,6) -- cycle;
	\draw[color=black] (-1,0) -- (9,0);
	\draw[color=black] (0,-1) -- (0,9);
	
	
	\node[above right=1pt of {(40/8,40/8)}, outer sep=2pt] {\small$(\frac{d}{d+1},\frac{d}{d+1})$};
	
	\node[below=2pt of {(6,0)}, outer sep=2pt] {\small(1,0)};
	
	\node[below left=2pt of {(0,6.5)}, outer sep=2pt] {\small(0,1)};
	
	\node[below left=2pt of {(0,0)}, outer sep=2pt] {\small(0,0)};
	
	\node[right=2pt of {(9,0)}, outer sep=2pt] {$\frac{1}{p}$};
	
	\node[left=2pt of {(0,9)}, outer sep=2pt] {$\frac{1}{q}$};
	\end{tikzpicture}
	\captionof{figure}{\color{black} Pairs $(\frac{1}{p}, \frac{1}{q})$ for which $T\cl L^p \tm L^q \ra L^1$. }\label{fig1}
\end{center} 


\subsection{Application to bilinear spherical averaging operator}
Recall that the bilinear spherical averaging operator is defined by
that the bilinear spherical averaging operator
\[
B(f,g)(x):= \int_{\sph{2d-1}}f(x - u_1)g(x - u_2)\dx{\sigma(u_1,u_2)},
\]
where $\sigma$ is the surface measure on the unit sphere in $\eun{2d}$. Multilinear spherical convolutions of this type were first introduced by Daniel Oberlin in the case where $d = 1$ \cite{OberlinMultilinearSpherical}. A complete characterization of $L^p$ bounds for these operators in the case of $d = 1$ was recently obtained by Shrivastava and Shuin \cite{ShrivastavaShuinMultilinearSpherical}. Here we address the case where $d \geq 2$ and show that $B\cl L^1 \tm L^1 \ra L^s$ for $s \in [1/2, 1]$. Jeong and Lee \cite{JeongLee} recently completely characterized the $L^p$ boundedness of the maximal version of the operator $B$ using a slicing technique; our approach in this section bears some resemblance to the slicing technique used by Jeong and Lee.

Let $d \geq 2$.
Let $\mc{D} := \{ (x, x) \cl  x \in \eun{d}\}$ be the diagonal subspace of $\eun{2d}$ and $\mc{A} := \{ (x, -x) \cl  x \in \eun{d}\}$ the antidiagonal subspace, and notice that these subspaces decompose $\eun{2d}$ orthogonally. Then for two points $(a, b), (c, d) \in \eun{d}\tm \eun{d} \simeq \eun{2d}$ we have $a - b = c - d$ if and only if $(a, b) - (c, d) \in \mc{D}$. 
Hence, if $\pi_{\mc{A}}$ is the orthogonal projection onto $\mc{A}$ and $\pi_{\mc{A}}(a, b) = (c, -c)$, then $a - b = 2c$. Thus for $E \se \eun{d}$, if $\pi_1\cl \eun{2d} \ra \eun{d}$ is the projection onto the first $d$ coordinates, i.e., $\pi_1(a,b) = a$, then 
\[
a - b \in E \quad \LRA \quad 2(\pi_1\circ\pi_{\mc{A}})(a,b) \in E. 
\]
Now let $R \in O(2d)$ be the orthogonal transformation with block matrix
\[
R = \frac{1}{\sqrt{2}}\begin{bmatrix}
I & -I \\
I & I
\end{bmatrix}
\]
that maps $\mc{A}$ onto the subspace $S_1 := \{(x, 0) \cl x \in \eun{d}\}$. Then
\[
a - b \in E \quad \LRA \quad 2(\pi_1\circ R)(a,b) \in E.
\]
But then by the invariance of the spherical measure under orthogonal transformations
\begin{align*}
\sigma\prs{\{ (a,b) \in \sph{2d-1} \cl a - b \in E \}}
%
&= \int_{\sph{2d-1}}1_E(a - b)\dx{\sigma(a,b)}\nlspace
&= \int_{\sph{2d-1}}1_{\frac{1}{2}E}[\pi_1(R(a,b))]\dx{\sigma(a,b)}\nlspace
&= \int_{\sph{2d-1}}1_{\frac{1}{2}E}[\pi_1(a,b)]\dx{\sigma(a,b)}\nlspace
&= \int_{\sph{2d-1}}1_{\frac{1}{2}E}(a)\dx{\sigma(a,b)}.
\end{align*}
For $dx$ the Lebesgue measure on $\eun{d}$ and $dy$ the Lebesgue measure on $\eun{d-1}$ we have
\begin{align*}
\int_{\sph{2d-1}}1_{\frac{1}{2}E}(a)\dx{\sigma(a,b)} &= 2\int_{B_{2d-1}(0,1)}1_{\frac{1}{2}E}(x)\frac{1}{\sqrt{1 - |x|^2 - |y|^2}}\dx{x}\dx{y}\nlspace
&= \int_{\eun{d}}1_{\frac{1}{2}E}(x)\cdot \prs{2\int_{\eun{d-1}}1_{B_{2d-1}(0,1)}(x,y)\frac{1}{\sqrt{1 - |x|^2 - |y|^2}}\dx{y}}\dx{x}\nlspace
&= \int_{\eun{d}}1_{\frac{1}{2}E}(x)\, F(x)\dx{x}. 
\end{align*}
Letting $r_0^2 := 1 - |x|^2$, we have
\begin{align*}
F(x) &= C_{d}\int_{0}^{{r_0}}\frac{r\os{d-2}\dx{r}}{\sqrt{ r_0^2 - r^2 }} \leq C
\end{align*}
for all $x \in B_d(0,1)$. Thus $F \in L^\nf(\eun{d})$, so the pushforward measure $\sigma_{(-)}$ is absolutely continuous with bounded density. Hence $B \cl L^1\tm L^1 \ra L\os{1}$ and by Proposition \ref{p2} $B \cl L^1\tm L^1 \ra L\os{1/2}$ for $d \geq 2$.


\section{$L^p$-Improving and Quasi-Banach Bounds for $k$-Simplex Operators for $d \geq k$}\label{s3}

In this section we establish $L^p$-improving and quasi-Banach bounds that hold in lower dimensions $d\geq k$, which are not included in the range of bounds obtained by the technique of Cook, Lyall, and Magyar \cite{CookLyallMagyar}.

Let $E_1\ldots E_k\se \eun{d}$ be measurable and using the symmetry of the operator assume WLOG that $|E_1| \leq |E_j|$ for all $j$. Then we have by the $L^p$-improving bounds for  spherical averages,
\begin{align*}
\norms{S^k(1_{E_1}, \ldots, 1_{E_k})}_{L\os{d+1}} &\leq \norm{S(1_{E_1})}_{L\os{d+1}}\nlspace
&\leq \norm{1_{E_1}}_{L\os{\frac{d+1}{d}}}\nlspace
&= |E_1|\os{\frac{d}{d+1}} \nlspace
&\leq |E_1|\os{\frac{d}{k(d+1)}}\cdots|E_k|\os{\frac{d}{k(d+1)}},
\end{align*}
so $S^k$ satisfies a restricted strong-type $(k + \frac{k}{d}, \ldots,  k + \frac{k}{d}, d+1)$ bound, which has an $L^p$ improvement ratio of $d$ versus the H\"older exponents. 

\begin{theorem}
	$S^k$ satisfies a restricted strong-type $(\frac{k(d+1)}{d}, \ldots, \frac{k(d+1)}{d}, d+1)$ bound for $d \geq k$.
\end{theorem}

\noindent Hence by interpolation against the $L^\nf \tm \cdots \tm L^\nf \ra L^\nf$ bound
\[
S^k \cl L^{p}\tm \cdot \tm L^p \ra L\os{\frac{dp}{k}} \quad \text{ for all } \quad p > \frac{k(d+1)}{d}.
\]
Now using the fact that each face of a regular $k$-simplex is a regular $(k-1)$-simplex, we have
\begin{align*}
\norms{S^k(f_1, \ldots, f_k)}_{L^1} &\leq \int_{\eun{d}}|f_1|(x)\int_{SO(d)}|f_2|(x - R(u_2 - u_1))\cdots|f_k|(x - R(u_k - u_1))\dx{\mu(R)}\dx{x}\nlspace
&= \int |f_1|(x)S\os{k-1}(|f_2|, \ldots, |f_k|)(x)\dx{x}\nlspace
&\leq \norm{f_1}_{L\os{\frac{dp}{dp - k+1}}}\norms{S\os{k-1}(|f_2|, \ldots, |f_k|)}_{L\os{\frac{dp}{k-1}}}\nlspace
&\leq C\norm{f_1}_{L\os{\frac{dp}{dp - k+1}}}\norm{f_2}_{L\os{p}}\cdots \norm{f_k}_{L^p},
\end{align*}
for all $p > \frac{(k-1)(d+1)}{d}$. Applying the technique from Section \ref{s1} then establishes nontrivial bounds for $S^k$.
\begin{corollary}
	The $k$-simplex operator $S^k$ satisfies the bound
	\[
	S^k \cl L\os{p_1}\tm\cdots \tm L\os{p_k} \ra L\os{r} \quad \text{ where } \quad \frac{1}{p_1} + \cdots + \frac{1}{p_k} = \frac{1}{r},
	\]
	and $(\frac{1}{p_1}, \ldots, \frac{1}{p_k}, \frac{1}{r})$ lies in the interior of the convex hull of the set of points $(\frac{1}{q_1}, \ldots, \frac{1}{q_k}, \frac{1}{r})$ with
	\[
	q_{\sigma(1)} = \frac{d+1}{d},  q_{\sigma(j)} = \frac{(k-1)(d+1)}{d} ~~~ \text{ for } \quad 2 \leq j \leq k, \quad r = \frac{d+1}{2d},
	\]
	for some permuation $\sigma$ of $\{1, \ldots, k\}$. In particular, 
	\[
	S^k \cl L\os{kr}\tm\cdots\tm L\os{kr} \ra L\os{r} ~~\quad \text{ for } \quad \quad r > \frac{d+1}{2d} \quad \text{ and } \quad d \geq k.
	\]
\end{corollary}
\noindent A straightforward calculation shows that for nice functions $f, g, h$,
\[
\angs{T(f,g), h} = \angs{f, T(g,h)}.
\]
It follows that $T\cl L^p \tm L^q \ra L^r$ implies $T\cl L\os{r'}\tm L\os{p}\ra L\os{q'}$ whenever $1 \leq r, q' < \nf$. Applying this with the $L^p$ improving bound above shows that $T\cl L^p \tm L^q \ra L^r$ for $d\geq 2$ when $(1/p, 1/q, 1/r)$ is one of the following $L^p$-improving triples
\[
\prs{ \frac{d}{d+1}, \frac{d}{2(d+1)}, \frac{d+2}{2d + 2} },\quad \prs{ \frac{d}{2(d+1)}, \frac{d}{d+1},  \frac{d+2}{2d + 2} }.
\]


\section{Restricted strong-type $(k,k,\ldots, k)$ bounds for $S^k$ for $d \geq k$}\label{s4}

In \cite{GreenleafIosevichKrauseLiu} the authors established that a family of operators that includes $T$ in dimension $d = 2$ satisfies a restricted strong-type $(2,2,2)$ bound. 
Here we adapt the ideas of the proof in \cite{GreenleafIosevichKrauseLiu} to obtain a restricted strong-type $(k,k,\ldots, k)$ bound for $S^k$ in dimensions $d \geq k$. The interesting cases occur when $d$ is close to $k$, since in higher dimensions this bound follows from the method of Cook, Lyall, and Magyar \cite{CookLyallMagyar}.
The key observation behind this adaptation is that if $p_1,\ldots, p_k$ are linearly independent points of $\sph{d-1}$, then on a neighborhood of $(p_1,\ldots,p_k)$ the addition map $(u_1,\ldots,u_k) \mapsto u_1 + \ldots + u_k$ from $\sph{d-1}\tm \cdots \tm \sph{d-1} \ra \eun{d}$ is a submersion and hence behaves locally like a projection.
\begin{theorem}
	$S^k$ is of restricted strong-type $(k, \ldots, k, k)$ in dimensions $d \geq k$. 
\end{theorem} 

\begin{proof}
	
We assume that $d\geq k$ and let $E_1, \ldots, E_k \se \eun{d}$ be measurable sets, WLOG (by the symmetry of the operator in its inputs) with $|E_1| \leq |E_2| \leq \ldots \leq |E_k|$. Our goal in this section is to show that
\begin{align*}
\norm{S^k(1_{E_1},\ldots, 1_{E_k})}_{L^k} &\leq C(\abs{E_1}\cdots \abs{E_k})\os{1/k},
\end{align*}
i.e., that $S^k$ is of restricted strong-type $(k, \ldots, k, k)$. 
We have for $\{0, u_1, \ldots, u_k\}$ the vertices of a regular $k$-simplex of unit side length, 
\begin{align*}
&\norms{S^k(1_{E_1}, \ldots, 1_{E_k})}_{L^k}^k
=  \int_{\eun{d}} 
\prod_{i=1}^{k}\prs{ \int_{O(d)} 1_{E_1}(x - R_iu_1)\cdots 1_{E_k}(x - R_iu_k) \dx\mu(R_i)}
\dx x.
\end{align*}
By the compactness of the product space $O(d)\tm \cdots \tm O(d) = (O(d))^k$ it is sufficient to show that for each $(R_1, \ldots, R_k) \in (O(d))^k$ there is a neighborhood $N(R_1, \ldots, R_k)$ of $(R_1, \ldots, R_k)$ such that, with $\mu^k := (\mu \tm \cdots \tm \mu)$,
\begin{align*}
\int_{\eun{d}} 
\int_{N(R_1,\ldots,R_k)}\prod_{i=1}^{k} 1_{E_1}(x - R_iu_1)\cdots 1_{E_k}(x - R_iu_k) \dx\mu^k(R_1, \ldots, R_k)
\dx x &\leq C|E_1|\cdots|E_k|,
\end{align*}
since then $(O(d))^k$ will be covered by finitely many such neighborhoods $N(R_1, \ldots, R_k)$.

To show that such a neighborhood exists, for each $i$ we will keep one of the factors in $$1_{E_1}(x - R_iu_1) \cdots 1_{E_k}(x - R_iu_k)$$ and drop the remaining $k-1$. Which factors we keep and which ones we drop will depend on the relative positions of the vectors $R_iu_j$. \\

\noindent\textit{(A) Selecting which factors to keep and drop:} We fix $(R_1, \ldots, R_k) \in (O(d))^k$ and
use the following algorithm to select which factors to keep and which to drop in each integral:
\begin{enumerate}[label=(\roman*)]
	\item For $i = 1$ we will keep $1_{E_1}(x - R_1u_1)$. We set $a_1 = u_1$ and $A_1 = E_1$.
	\item Suppose that $j < k$ and we have chosen $a_1, \ldots a_j$, with $a_i = u_m$ for some $m\leq i$ for each $i$ and such that $R_1a_1,\ldots, R_ja_j$ are linearly independent. 
	We choose $a_{j+1}$ and $A_{j+1}$ as follows: Note that $\{0, R_{j+1}u_1, \ldots, R_{j+1}u_{j+1}\}$ form the set of vertices of a regular $(j+1)$-simplex, and hence the vertices $R_{j+1}u_1, \ldots, R_{j+1}u_{j+1}$ are linearly independent. It follows that there must be a $p$ in $1, \ldots, j+1$ such that $R_{j+1}u_{p}\not\in\sspan\{ R_1a_1, \ldots, R_ja_j \}$. Then we set $a_{j+1} = a_p$ and $A_{j+1} = A_p$. 
\end{enumerate}
This algorithm produces sequences $a_1, \ldots, a_k$ of vectors and $A_1, \ldots, A_k$ of sets, where each $a_i$ is equal to some $u_m$, and $A_i = E_m$, with $m \leq i$ for $i = 1, \ldots, k$. For each $i$ we will keep the factor $1_{A_i}(x - R_ia_i)$ in the integrand and drop the remaining $k-1$ factors corresponding to $R_i$. \\

\noindent\textit{(B) Bounding inner integral by parameterized spherical integral:}
We now let 
\[
B:= B(R_1a_1, r)\tm \cdots\tm B(R_ka_k, r),
\]
where $r > 0$ will be chosen sufficiently small in a later step, and define the open neighborhood 
\[
N(R_1, \ldots, R_k) := \{ (S_1,\ldots, S_k) \in (O(d))^k \cl (S_1a_1,\ldots,S_ka_k) \in B \}.
\]
We now have, denoting again $d\mu^k := d(\mu\tm\cdots\tm\mu)$,
\begin{align*}
&\int_{\eun{d}} 
{\int_{N(R_1,\ldots, R_k)}\prod_{i=1}^{k} 1_{E_1}(x - R_iu_1)\cdots 1_{E_k}(x - R_iu_k) \dx\mu^k}(R_1, \ldots, R_k)
\dx x\nlspace
\leq &\int_{\eun{d}}\int_{N(R_1,\ldots,R_k)}1_{A_1}(x - R_1a_1)\cdots 1_{A_k}(x - R_ka_k)\dx\mu^k(R_1,\ldots,R_k)\dx x\nlspace
= &\int_{\eun{d}}\int_{B\se (\sph{d-1})^k}1_{A_1}(x - p_1)\cdots 1_{A_k}(x - p_k)
\dx\sigma^k(p_1,\ldots,p_k)\dx x\nlspace
= &\int_{\eun{d}}1_{A_1}(x)
\int_{B\se (\sph{d-1})^k}1_{A_2}(x - (p_2 - p_1))\cdots 1_{A_k}(x - (p_k - p_1))\dx\sigma^k(p_1,\ldots,p_k)\dx x,\numberthis\label{e4}
\end{align*}
and it suffices to show that inner integral in the last line is $\leq C|A_2|\cdots|A_k|$ with constant $C$ independent of $E_1,\ldots, E_k$ and $x$. 

There is an $s = s(r) > 0$ and for each $i$ an orthogonal transformation $O_i$ such that a subset $S$ of $\sph{d-1}$ containing $B(R_ia_i, r) \cap\sph{d-1}$ is parameterized by 
\[
f_i(x_1,\ldots, x_{d-1}) = O_i(x_1,\ldots, x_{d-1}, \sqrt{1 - |x|^2}) \quad \text{ for } \quad x = (x_1,\ldots, x_{d-1}) \in C_{d-1}(0, s),
\]
where $f_i(0) = R_ia_i$ and $C_{d-1}(0,s) = [-s,s]\os{d-1}$. 
Recall that the tangent space to $\sph{d-1}$ at the point $p_i := R_ia_i$ can be realized as the hyperplane $P(p_i):=\{ x \cl x\cdot p_i = 0 \}$ and that it is spanned by the partial derivatives of $f_i$ at 0. By choosing $s$ (and hence $r$) small enough we can assume that the ``volume element'' of the coordinate chart $f_i$ is bounded on $C_{d-1}(0,s)$, with bound depending only on $s$. By translation invariance of Lebesgue measure we can assume that $x = 0$, and we get for $y^i = (y^i_1,\ldots, y^i_{d-1})$,
\begin{align*}
&\int_{B\se (\sph{d-1})^k}1_{A_2}(p_2 - p_1)\cdots 1_{A_k}(p_k - p_1)\dx\sigma^k(p_1,\ldots,p_k) \nlspace
\leq ~C &\int_{C_{k(d-1)}(0,s)}1_{A_2}(f_2(y^2) - f_1(y^1))\cdots1_{A_k}(f_k(y^k) - f_1(y^1))\dx(y^1,\ldots, y^k).\\[-2mm]
\end{align*}

\noindent\textit{(C) Linear algebra for tangent hyperplanes:}
Recall that the vectors $p_i = R_ia_i$, $1 \leq i \leq k$ were chosen to be linearly independent. For each $i \geq 2$ let 
\[
p_i = p_i\os{\parl} + p_i\os{\perp} \quad \text{ with } \quad  p_i\os{\parl} \in P(p_1) \quad \text{ and } \quad p_i\os{\perp} \in \sspan\{p_1\}.
\]
If there are $\alpha_1,\ldots\alpha_k$ not all zero with $0 = \alpha_1 p_1\os{\parl} + \ldots + \alpha_k p_k\os{\parl}$, then
\[
\alpha_1 p_1 + \ldots + \alpha_k p_k = \alpha_1 p_1\os{\perp} + \ldots + \alpha_k p_k\os{\perp} = cp_1,
\]
contradicting the linear independence of $p_1, \ldots, p_k$. Hence the orthogonal projections $p_i\os{\parl}$ are also linearly independent. Further, since $p_i$ and $p_1$ are not linearly dependent for $i \geq 2$ the vector $p_i\os{\parl}$ is not contained in $P(p_i)$. Formally, if $p_i\os{\parl} \in P(p_i)$, then
\begin{align*}
\norm{p_i}^2 = p_i \cdot p_i = p_i \cdot p_i\os{\perp} \leq \norm{p_i}\norm{p_i\os{\perp}} \quad \Ra \quad \norm{p_i} = \norm{p_i\os{\perp}} \quad \Ra \quad p_i = c p_1,
\end{align*}
contradicting the independence of $p_i$ and $p_1$. Hence $\sspan\prss{P(p_i)\cup \{p_i\os{\parl}\}} = \eun{d}$.

Now since $p_2\os{\parl}, \ldots, p_k\os{\parl}$ are in the image of $Df_1(0)$, which is exactly $P(p_1)$, there is an invertible $(d-1)\tm (d-1)$ matrix $A$ such that the first columns of $Df_1(0)A = D(f_1\circ A)(0)$ are $p_2\os{\parl}, \ldots, p_k\os{\parl}$. Then by a simple change of the first $d-1$ variables we have for some bounded open set $V$,
\begin{align*}
&\int_{C_{k(d-1)}(0,s)}1_{A_2}(f_2(y^2) - f_1(y^1))\cdots1_{A_k}(f_k(y^k) - f_1(y^1))\dx(y^1,\ldots, y^k)\nlspace
\leq ~C&\int_{V\tm C_{(k-1)(d-1)}(0,s)}1_{A_2}(f_2(y^2) - f_1(Ay^1))\cdots1_{A_k}(f_k(y^k) - f_1(Ay^1))\dx(y^1,\ldots, y^k)\nlspace
\leq ~C&\int_{C_{k(d-1)}(0,ks)}1_{A_2}(f_2(y^2) - f_1(Ay^1))\cdots1_{A_k}(f_k(y^k) - f_1(Ay^1))\dx(y^1,\ldots, y^k),
\end{align*}
where $k$ only depends on our choice of $a_1, \ldots, a_k$.\\

\noindent\textit{(D) Applying the inverse function theorem:} We consider the map 
\[
(y^1,\ldots, y^k) \mapsto \Phi(y^1,\ldots, y^k) =\prs{f_2(y^2) - f_1(Ay^1),\, f_k(y^k) - f_1(Ay^1)}
\]
from $\eun{k(d-1)}$ into $\eun{(k-1)d}$. Since each map $f_i$ is a submersion, i.e., each has a surjective derivative at each point, by the construction of this map the partial derivatives $\pd_{y^i_j}\Phi(0)$ for $i \geq 2$ form a linearly independent set $\mc S$ of $(k-1)(d-1)$ vectors. By the arguments in (C) above the partial derivatives $\pd_{y^1_j}\Phi(0)$ are also linearly independent with the derivatives in $\mc{S}$ and with one another. Thus we have a set of $(k-1)d$ linearly independent partial derivatives at the origin, specifically the derivatives $\{ \pd_{y^1_j}\Phi(0) \cl i\geq 2 \text{ or } (i = 1 \text{ and } 1\leq j \leq k-1) \}$ are linearly independent. 

If we denote $y^1_f = (y^1_1,\, \ldots,\, y^1_{k-1})$ and $y^1_l = (y^1_{k},\, y^2_{d-1})$, then the inverse function theorem tells us that we can choose $s$ (and hence $r$) small enough that for each fixed $y^1_l \in C_{d-k}(0,ks)$ the map
\[
(y^1_{f}, y^2, \ldots, y^k) \mapsto \Phi(y^1,\ldots, y^k) =\prs{f_2(y^2) - f_1(Ay^1),\, f_k(y^k) - f_1(Ay^1)}
\]
is a diffeomorphism of $C_{(k-1)d}(0,ks)$ onto an open subset $U(y^1_l)$ of $\eun{(k-1)d}$. 
Note that by choosing $s$ small enough we can assume that the Jacobian of this diffeomorphism is uniformly bounded for all choices of $y^1_l \in C_{d-k}(0,ks)$.
Then we have
\begin{align*}
&\int_{C_{k(d-1)}(0,ks)}1_{A_2}(f_2(y^2) - f_1(Ay^1))\cdots1_{A_k}(f_k(y^k) - f_1(Ay^1))\dx(y^1,\ldots, y^k)\nlspace
\leq~ C&\int_{C_{d-k}(0,ks)}\int_{U(y^1_l)}1_{A_2}(z^2)\cdots1_{A_k}(z^k)\dx(z^2,\ldots, y^k)\dx y^1_l\nlspace
\leq~ C&|A_2|\cdots |A_k| \leq C |E_2|\cdots |E_k|,\\[-2mm]
\end{align*}
where we have used that $A_i = E_m$ with $m \leq i$ and the assumption that $|E_1| \leq |E_2| \leq \ldots \leq |E_k|$. Inserting this into \eqref{e4} gives the required estimate.\\

\noindent\textit{(E) Combining estimates:} Inserting the previous estimate back into \eqref{e4}, using compactness to get a finite covering, summing the corresponding integrals and taking the $k$th root gives
\begin{align*}
\norm{S^k(1_{E_1},\ldots, 1_{E_k})}_{L^k} &\leq C(\abs{E_1}\cdots \abs{E_k})\os{1/k}.
\end{align*}
\end{proof}




\vspace{1mm}

\begin{bibdiv}
	\begin{biblist}
		\bib{CookLyallMagyar}{article}{
			author = {Brian Cook and Neil Lyall},
			author = {Akos Magyar},
			title = {Multilinear maximal operators associated to simplices},
			journal = {Journal of the London Mathematical Society},
			volume = {n/a},
			number = {n/a},
			pages = {},
			keywords = {42B25 (primary)},
			doi = {https://doi.org/10.1112/jlms.12467},
			url = {https://londmathsoc.onlinelibrary.wiley.com/doi/abs/10.1112/jlms.12467},
			eprint = {https://londmathsoc.onlinelibrary.wiley.com/doi/pdf/10.1112/jlms.12467},
		}
		\bib{GrafakosKalton}{article}{
			author={Grafakos, Loukas},
			author={Kalton, Nigel},
			title={Some remarks on multilinear maps and interpolation},
			journal={Math. Ann.},
			volume={319},
			date={2001},
			number={1},
			pages={151--180},
			issn={0025-5831},
			review={\MR{1812822}},
			doi={10.1007/PL00004426},
		}
		\bib{GreenleafIosevichKrauseLiu}{article}{
			Author = {Allan Greenleaf and Alex Iosevich and Ben Krause and Allen Liu},
			Title = {Bilinear generalized Radon transforms in the plane},
			Year = {2017},
			Eprint = {arXiv:1704.00861},	
		}
		\bib{GreenleafIosevichTriangle}{article}{
			author={Greenleaf, Allan},
			author={Iosevich, Alex},
			title={On triangles determined by subsets of the Euclidean plane, the
				associated bilinear operators and applications to discrete geometry},
			journal={Anal. PDE},
			volume={5},
			date={2012},
			number={2},
			pages={397--409},
			issn={2157-5045},
			review={\MR{2970712}},
			doi={10.2140/apde.2012.5.397},
		}
		\bib{JeongLee}{article}{
			author={Jeong, Eunhee},
			author={Lee, Sanghyuk},
			title={Maximal estimates for the bilinear spherical averages and the
				bilinear Bochner-Riesz operators},
			journal={J. Funct. Anal.},
			volume={279},
			date={2020},
			number={7},
			pages={108629, 29},
			issn={0022-1236},
			review={\MR{4103874}},
			doi={10.1016/j.jfa.2020.108629},
		}
		\bib{LaceySparseSphericalMaximal}{article}{
			author={Lacey, Michael T.},
			title={Sparse bounds for spherical maximal functions},
			journal={J. Anal. Math.},
			volume={139},
			date={2019},
			number={2},
			pages={613--635},
			issn={0021-7670},
			review={\MR{4041115}},
			doi={10.1007/s11854-019-0070-2},
		}
		\bib{OberlinMultilinearSpherical}{article}{
			author={Oberlin, Daniel M.},
			title={Multilinear convolutions defined by measures on spheres},
			journal={Trans. Amer. Math. Soc.},
			volume={310},
			date={1988},
			number={2},
			pages={821--835},
			issn={0002-9947},
			review={\MR{943305}},
			doi={10.2307/2000993},
		}
		\bib{PalssonSovine}{article}{
			author={Palsson, Eyvindur A.},
			author={Sovine, Sean R.},
			title={The triangle averaging operator},
			journal={J. Funct. Anal.},
			volume={279},
			date={2020},
			number={8},
			pages={108671, 21},
			issn={0022-1236},
			review={\MR{4109092}},
			doi={10.1016/j.jfa.2020.108671},
		}
		\bib{ShrivastavaShuinMultilinearSpherical}{article}{
			Author = {Saurabh Shrivastava and Kalachand Shuin},
			Title = {$L^p$ estimates for multilinear convolution operators defined with spherical measure},
			Year = {2020},
			Eprint = {arXiv:2006.03754},
		}
		\bib{StovallThesis}{article}{
			Author = {Stovall, Betsy},
			Title = {Lp inequalities for certain generalized Radon transforms},
			journal = {dissertation},
			Year = {2009},
		}	
	\end{biblist}
\end{bibdiv}
\end{document}